\documentclass[a4paper,leqno]{article}
\usepackage{amstext}
\usepackage{amsmath}
\usepackage{amssymb}
\usepackage{amsthm}
\usepackage{color}
\usepackage{xspace}
\usepackage{hyperref}
\usepackage{amsrefs}
\usepackage{xypic}
\usepackage{stmaryrd}
\theoremstyle{plain}
\newtheorem{Def}{Definition}
\newtheorem{Thm}{Theorem}
\newtheorem{Rmk}[Thm]{Remark}

\newtheorem{Cor}[Thm]{Corollary}
\newtheorem{Lem}[Thm]{Lemma}
\newtheorem{Pro}[Thm]{Proposition}

\newcommand{\NN}{\mathbb{N}}

\newcommand{\FANF}{\textrm{FAN}\ensuremath{_{\mathrm{full}}}\xspace}
\newcommand{\FAND}{\textrm{FAN}\ensuremath{_{\Delta}}\xspace}
\newcommand{\FANP}{\textrm{FAN}\ensuremath{_{\Pi^{0}_{1}}}\xspace}
\newcommand{\FANc}{\textrm{FAN}\ensuremath{_{c}}\xspace}

\newcommand{\ext}[1]{\llbracket #1 \rrbracket}

\begin{document}
\title{Separating the Fan Theorem and Its Weakenings}

\author{Robert S. Lubarsky and Hannes Diener \\ Dept. of Mathematical
Sciences \\ Florida Atlantic University \\ Boca Raton, FL 33431 \\
Robert.Lubarsky@alum.mit.edu \\ Department Mathematik, Fak. IV \\
Emmy-Noether-Campus, Walter-Flex-Str. 3 \\ University of Siegen \\
57068 Siegen, Germany \\ diener@math.uni-siegen.de} \maketitle
\begin{abstract}
Varieties of the Fan Theorem have recently been developed in
reverse constructive mathematics, corresponding to different
continuity principles. They form a natural implicational
hierarchy. Some of the implications have been shown to be strict,
others strict in a weak context, and yet others not at all, using
disparate techniques. Here we present a family of related Kripke
models which separates all of the as yet identified fan theorems.
\\ {\bf keywords:} fan theorems, Kripke models, forcing,
non-standard models, Heyting-valued models\\{\bf AMS 2010 MSC:}
03C90, 03F50, 03F60, 03H05
\end{abstract}

\section{Introduction}

To be able to talk about fans, Cantor space, and similar objects
properly, we will start by introducing some notation. The space of
all infinite binary sequences, endowed with the standard topology
(wherein a basic open set is given by a finite binary sequence),
will be denoted by $2^{\mathbb{N}}$; the set of all finite binary
sequences will be denoted by $2^{\ast}$.  The concatenation of
$u,v \in 2^{\ast}$ will be denoted by $u \ast v$. For $\alpha \in
2^{\mathbb{N}}$ and $n \in \mathbb{N}$, the first $n$ elements of
$\alpha$ form a finite sequence denoted by $\overline{\alpha} n$.
A subset $B \subseteq 2^{\ast}$ is called a {\bf bar} if
\[\forall \alpha \in 2^{\mathbb{N}} \exists n \in \mathbb{N}
(\overline{\alpha}n \in B ), \]
and a bar is called {\bf uniform} if
\[\exists n \in \mathbb{N} \; \forall \alpha \in 2^{\mathbb{N}} \; \exists
m \leqslant n \; (\overline{\alpha}m \in B ). \]

Notice that if a bar $B$ is closed under extensions, that is if
\[ \forall u \in 2^{\ast} (u \in B \implies \forall v \in 2^{\ast}  \; u\ast
v \in B), \] then it is uniform if and only if \[\exists n \in
\mathbb{N} \; \forall \alpha \in 2^{\mathbb{N}} \;
(\overline{\alpha}n \in B ). \] Not all of the bars we consider
will be closed under extensions.

There are currently four versions of Brouwer's Fan Theorem in
common use. All of them enable one to conclude that a given bar is
uniform. The differences among them lie in the definitional
complexity demanded (as an upper bound) of the bar in order for
the theorem to apply to it, which ranges from the very strongest
requirement to no restriction on the bar at all. A bar $C \subset
2^ \ast $ is {\bf decidable} if it is decidable as a set:
    \[\forall u \in 2^* \; u \in C \vee u \not \in C.\]
A bar $C \subset 2^\ast $ is called a {\bf $c$-bar} if there
exists a decidable set $C' \subset 2^\ast $ such that
    \[ u \in C \iff \forall v \in 2^\ast \, \left(u \ast v \in C'\right).\]
A bar $B \subset 2^\ast $ is called a {\bf ${\mathbf
\Pi_{1}^0}$-bar} if there exist a decidable set $S \subset
2^{\ast} \times \NN$ such that
\[u \in B \iff \forall n (u,n) \in S \ .  \]
(The $\Pi^0_n$-nomenclature alludes to the arithmetical hierarchy
in computability theory.) We can now state four commonly used
versions of the Fan Theorem.
\begin{quote}
\begin{tabular}{ll}
    \FAND: & Every decidable bar is uniform. \\
    \FANc: &Every $c$-bar is uniform. \\
    \FANP: & Every $\Pi^0_1$-bar is uniform. \\
    \FANF: & Every bar is uniform.
\end{tabular}
\end{quote}

Notice that every decidable bar can be taken to be closed under
extensions; that is, the closure of a decidable bar under
extension is still decidable. If there is no restriction on the
definability of a bar, then every bar can be taken to be so
closed, by working with the closure of any given bar. Every
$c$-bar is already closed under extension. In contrast,
$\Pi^0_1$-bars seemingly cannot be replaced by their closures
while remaining $\Pi^0_1$.

By way of motivation, these principles were developed within
reverse constructive mathematics, because they are equivalent with
certain continuity principles. In particular, over a weak base
theory, \FAND is equivalent with the assertion that every
uniformly continuous, positively valued function from [0,1] to
$\mathbb{R}$ has a positive infimum~\cite {JR}, \FANc with the
uniform continuity of every continuous $f : 2^\NN \rightarrow \NN$
\cite {jB06}, and \FANP with the uniform equicontinuity of every
equicontinuous sequence of functions from [0,1] to $\mathbb{R}$
\cite {DL}.

The following implications hold trivially \cite{jB06,hD08b} and
over a weak base theory:

\begin{equation*} \label{Equ:impl2}
\FANF \implies \FANP \implies \FANc \implies \FAND.
\end{equation*}
One naturally wonders whether any of the implications can be
reversed, including whether \FAND is outright provable in
constructive set theory. Some such non-implications have already
been determined.

\begin {itemize}
\item It is well-known (see \cite{Bees} for instance) that \FAND is
not provable, via recursive realizability. That is, there is an
infinite (Turing) computable sub-tree of $2^\ast$ with no infinite
computable branch, which fact translates to a failure of \FAND
under IZF (Intuitionistic ZF, the constructive correlate to
classical ZF) via recursive realizability, and also to the
independence of WKL (Weak K\"onig's Lemma) over RCA$_0$ in reverse
mathematics \cite{SS}.

\item Berger \cite{jB09} shows that \FAND does not imply \FANc over
a very weak base system. His argument is in its essence a
translation of the reverse mathematics proof that WKL$_0$ does not
imply ACA \cite {SS}, by coding the Turing jump into a $c$-bar. In
order for this argument to work, he must be in a context in which
the existence of the Turing jump is not outright provable, hence
the use of a weak base system.

\item Fourman and Hyland \cite {mF79} present a Heyting-valued,
almost topological, model in which \FANF fails; they do not
address which fragments of the Fan Theorem might hold, since these
distinctions were not available at the time. We show below that
\FANP holds in their model, separating the left-most pair of
principles in the diagram above.

\end {itemize}

We are not aware of any prior proofs separating \FANc and \FANP.

The goal of this paper is to separate all of these principles via
a uniform technique. This has several benefits. For one, it
separates \FANc and \FANP. For another, it separates \FAND and
\FANc over full IZF. That is new because Berger's argument still
leaves open the possibility that IZF would allow that implication
to go through; independence of \FANc over IZF + \FAND means that
it does not. In addition, since the arguments employed rather
handily provide four separation results, they seem to provide a
flexible tool that might be useful elsewhere. This seems not to be
the case for the other techniques that have been used. It could
well be the case, for instance, that realizability could produce
all of the results discussed here. But no one has been able to do
this yet. As for the Fourman-Hyland argument, they also show in
the same work that all topological models satisfy \FANF. So for
the separations of interest here, topological models are just out.
To be sure, variants of topological models, along the lines used
by Fourman and Hyland, might still do the trick. But before coming
up with the arguments below that's exactly what we tried, and got
nowhere. In short, we cannot say that the techniques used here get
you anything that could not be gotten by other means, but at least
it seems to be easier to use. Beyond that, it could be the case
that the proofs here really are in some sense the right ones for
these results. In the face of the perfectly nice realizability and
Heyting-valued models that provide some of these separations, we
are not at this point making that claim. While the constructions
below are natural enough, they are not so compelling as to seem
canonical. Nonetheless, since they seem to work so well, it might
be that with further reflection and development, it turns out that
proofs along these lines are the way to go for a large class of
problems.

As for what the techniques employed actually are, we would like to
provide some motivation for how we happened upon them. Since it
seemed that realizability and Heyting algebras weren't working, we
turned to the only other kind of model we know of, Kripke models.
To build a tree we could control, along with its paths, over set
theory with full Separation and Collection, we turned to forcing.
In order to have the trees be decidable, yet not completely pinned
down, as required by the theories in question, we were forced to
use non-standard integers, to provide non-standard levels on the
trees.

Since this is a paper about constructive mathematics, a word about
the meta-theory used is in order. It is classical through and
through. We work in ZFC. Presumably most if not all of the
arguments are fully constructive, as in so many mathematical
papers in all fields. We did not check, and so have no idea.

In the next section we discuss the Fan Theorem in topological and
related models, including giving a proof that the Fourman-Hyland
model satisfies \FANP. The following sections provide the
advertised separation results, going right-to-left in the diagram
above. We then close with some questions.

\section{The Fan Theorem in Heyting-Valued Models}

To make this paper somewhat self-contained, we repeat the proof
that explains why the construction afterwards is more complicated
than just a topological model.

\begin{Pro} (Fourman-Hyland \cite{mF79})
In any topological model \FANF holds.
\end{Pro}
\begin{proof}
Let $T$ be a topological space, and suppose
\begin {center} $T
\Vdash ``B \subseteq 2^{\ast}$ is a bar closed under extension."
\end {center}
Then, in particular, for any external sequence $\alpha \in
2^{\NN}$ (that is, one from the ground model)
\begin{equation} \label{Eqn:bar}
T =\ext{\exists n \; \overline{\alpha}n \in B } = \bigcup_{n \in
\NN} \ext{\overline{\alpha}n \in B}\ .
\end{equation}
Let $A_{u} $ denote the open set
\[ \ext{\text{the bar } \{w \mid  u*w \in B\} \text{ is uniform}} \ . \]
If $T \not \Vdash ``B$ is uniform," then choose some $p \notin
A_{()}$. Define a set $Tr =\{u \in 2^{\ast} \mid p \not \in A_{u}
\}$. Since  $A_{u} = A_{u*0} \cap A_{u*1}$ for any $u \in 2^{*}$,
$Tr$ is a tree (i.e. closed under restriction) with no terminal
nodes. Since in addition $() \in Tr$ (that is, $Tr$ is non-empty),
$Tr$ is infinite. Thus, by Weak K\"onig's Lemma, there exists an
infinite path $\beta$ in $Tr$. By the definition of $Tr,$ $p
\notin A_{\overline{\beta}n}$ for all $n \in \NN$. Now Equation
\ref{Eqn:bar} yields the existence of  $n \in \NN$ such that
\[ p \in \ext{\overline{\beta}n \in B} \ ; \]
but this contradicts
$\ext{\overline{\beta}n \in B}  \subset A_{\overline{\beta}n}$.
\end{proof}

This suggests that if we are looking for models in which some form
of the Fan Theorem fails we need to ``delete points''. This was
done in \cite{mF79}, section 4, where they consider $K(T)$, the
coperfect open sets of a topological space $T$. This can be viewed
as the equivalence classes of the open sets of $T$, under which an
open set is identified with its smallest coperfect superset. In
this setting, removing a point from an open set does not change
the set.

\begin{Def} A Heyting algebra is connected if $A \vee B = \top$ and
$A \wedge B = \bot$ implies that either $A = \top$ or $A= \bot$.
\end{Def}

Let $\Omega$ be $K([0,1] \times [0,1])$. It is easy to see that
$\Omega$ is connected.

\begin{Pro} If $H$ is a connected Heyting algebra, then $H \Vdash
\FANP.$
\end{Pro}
\begin{proof}
Suppose $H \Vdash ``B$ is a $\Pi^0_1$-bar, given say by $S: u \in
B$ iff $\forall n \in \NN \; (u,n) \in S."$ Since $H \Vdash ``S$
is decidable," for any $u \in 2^*$ and $n \in \NN,$

\[ H \Vdash ``(u,n) \in S \vee (u,n) \not \in S." \]
By the connectedness of $H$ either $H \Vdash ``(u,n) \in S"$ or $H
\Vdash ``(u,n) \not \in S."$ So define a set $\tilde{B} \subset
2^{\ast}$ in the metatheory by
\[ u \in \tilde{B} \iff \forall n \in \NN \; H \Vdash ``(u ,n)\in
S."\] $\tilde{B}$ is itself a bar, as follows. Let $\alpha \in
2^{\NN}$ be arbitrary. If $\overline{\alpha} n \notin \tilde{B}$
for all $n \in \NN$ then for all $n$ there exist $i_{n}$ such that
$ \llbracket (\overline{\alpha}n ,i_{n})\in B \rrbracket = \bot$.
Thus
\[  \llbracket \forall m \in \NN \; (\overline{\alpha}n ,m)\in B
\rrbracket = \bot\] for any $n \in \NN$, and therefore
\[ \llbracket \exists n \in \NN \forall m \in \NN \;
(\overline{\alpha}n ,m)\in B \rrbracket = \bot  \ ;\] a
contradiction to $B$ being a bar internally. Hence $\tilde{B}$ is
a bar externally, and therefore, working with a classical
metatheory (or simply the Fan Theorem), it is uniform. So there
exists $N$ such that for all $u$ of length $N$ some initial
segment of $u$ is in $\tilde{B}$. Then it is easy to see, that
this same $N$ witnesses the uniformity of $B$ internally.
\end{proof}

\begin{Cor}
\FANP does not imply \FANF (over IZF).
\end{Cor}
\begin{proof}
In \cite{mF79} it shown that $ \Omega \not \Vdash \FANF$.
\end{proof}

\section {\FAND is not Provable}
As discussed in the introduction, recursive realizability shows
that IZF does not prove \FAND. However, we do not see how to adapt
that, or the Heyting-valued model from the previous section, to
the other desired separation results. Hence we are hoping not
merely to provide here a different model falsifying \FAND as a
technical exercise, but rather to provide a technique more
flexible than those referenced, to produce the other separation
results. Of course, if this really is a flexible technique, it
should work for the known separations too.

We will build a Kripke model, working within ZFC. To construct a
bar, it will be crucial to control what paths exist. This is most
easily done with a generic set, in the sense of forcing.

\begin {Def} Let the forcing partial order {\it P} be
the set of appropriate labelings of finitely many nodes from
2$^*$. A labeling of nodes assigns to each one either IN, OUT, or
$\infty$, with the following restrictions. Any node labeled IN has
no descendant, the idea being that once a node gets into the
eventual bar so are all of its descendants automatically, so
nothing more need be said. Any descendant of a node labeled OUT
must be labeled IN or OUT. Finally, for any node labeled $\infty$,
if both children are labeled, then at least one of them must be
labeled $\infty$. \end {Def}

Let $G$ be a generic through the condition that labels $\langle
\rangle$ with $\infty$. By straightforward density arguments, any
node labeled OUT by $G$ has a uniform bar above it (or below it,
depending on how you draw your trees) all labeled IN, and every
node labeled $\infty$ has a path through it always labeled
$\infty$, in fact a perfect set of such.

Let $B = \{ \alpha \in 2^* |$ for some $n \; G(\alpha
\upharpoonright n)$ = IN\}. $B \in M[G]$ is the interpretation
$\sigma_B^G$ of the term $\sigma_B = \{ \langle p, \hat{\alpha}
\rangle \mid $ for some $n \; p(\alpha \upharpoonright n)$ = IN\}.
(As usual, the function $\hat{^.}$ is the canonical injection of
the ground model into the terms: $\hat{x} = \{ \langle \emptyset,
\hat{y} \rangle \mid y \in x \}$.) Because of these latter
$\infty$-paths, $B$ is not a bar. However, we might reasonably
think that if we no longer had access to the distinction between
the OUT and the $\infty$ nodes, we might no longer be able to
build a path avoiding $B$. This intuition is confirmed by the next
proposition.

\begin {Def} The shadow forcing $Q$ is the set of functions from
finite subtrees of $2^*$ to \{IN, OUT\} such that any node labeled
IN has no descendant. Equivalently, $Q$ is the sub-partial order
of $P$ beneath the condition labeling $\langle \rangle$ with OUT
(together with the condition which labels $\langle \rangle$ IN,
which has no extension). The canonical projection proj$_Q$ of $P$
onto $Q$ replaces all occurrences of $\infty$ with OUT. The
canonical projection of the terms of $P$'s forcing language to
those of $Q$'s, ambiguously also called proj$_Q$, acts by applying
proj$_Q$ to the conditions that appear in the terms, hereditarily.
(Notice that $Q$ term are also $P$ terms.)
\end {Def}

Notice that a $P$-filter projects to a $Q$-filter. If $G$ is a
generic $P$-filter, then proj$_Q(G)$ will not be $Q$-generic,
because in $Q$ the terminal conditions are dense. Still,
proj$_Q(G)$ induces an interpretation $\sigma^{proj_Q(G)}$ of the
terms $\sigma$ of $Q$. These interpretations are in $M[G]$, as
they are easily definable from $\sigma$ and $G$; alternatively,
$\sigma^{proj_Q(G)} = (proj_Q^{-1 \prime \prime}\sigma)^G.$

For any $P$-filter $G, proj_Q(\sigma_B)^{proj_Q(G)} = B$: the
induced interpretation of the projection of $B$ is just $B$
itself. Effectively, $B$ as a $P$-term is already a $Q$-term.

\begin {Pro} If $\sigma$ is a $Q$-term and $p \Vdash_P
``proj_Q^{-1\prime\prime}\sigma$ is an infinite branch through
$2^*$," then $p \Vdash_P ``proj_Q^{-1\prime\prime}\sigma$ goes
through $\sigma_B$."
\end {Pro}

\begin {proof} By standard forcing technology, it suffices to
extend $p$ to some condition forcing
$``proj_Q^{-1\prime\prime}\sigma$ goes through $\sigma_B$," as
then it will be dense beneath $p$ to force as much, and so will
happen generically.

First extend $p$ so that every sequence in 2$^*$ of length
2$^{n-1}$ for some $n$ either is labeled OUT or $\infty$ or has a
proper initial segment labeled IN. Then extend again by adjoining
both children to all nodes of length 2$^{n-1}$, and labeling them
$\infty$ whenever possible (otherwise IN or OUT). For a technical
reason soon to become clear, we must extend yet again. This time
have the domain include all length $k$ descendants of the length
$n$ nodes not labeled IN, and label them so that every length $n$
node labeled $\infty$ has a unique descendant of length $k$
labeled $\infty$, and, most importantly, for each pair of nodes
$\alpha$ and $\beta$ of length $k$ labeled $\infty$, there is some
$i$ with $\alpha(i) = 1$ and $\beta(i) = 0.$ One way of doing this
is to let $s$ be the number of nodes of length $n$ labeled
$\infty$, to let $k$ be $n+s$, and to build the $\infty$-labeled
descendant of the $j^{th}$ such node by adjoining to it $j-1$ 0's,
a 1, and then $s-j$ 0's, all other descendants of length $k$ being
labeled OUT.

Extend one last time to $q \Vdash
proj_Q^{-1\prime\prime}\sigma(\hat{k}) = \hat{\alpha}$ for some
fixed $\alpha$, where as usual $\hat x$ is the standard term for
the internalization of the set $x$. Moreover, $q$ should force the
equality in the strong sense that for each $j < k$ there is a term
$\tau$ and a condition $r \geq q$ with $\langle r, \tau \rangle
\in proj_Q^{-1\prime\prime}\sigma$ and $q \Vdash \tau = \langle
\hat j, \hat \alpha(\hat j) \rangle$; even further, if $\alpha(j)
= 1$ then $q$ forces a particular element to be in $\tau$'s second
component.

If $q$ labels some initial segment of $\alpha$ IN then we're done.

If $q$ labels $\alpha$ OUT then it is dense beneath $q$ that all
descendants of $\alpha$ of some fixed length are labeled IN, and
again we're done.

If $q$ labels $\alpha \; \infty$ then let $q_{alt}$ be identical
to $q$ except that all descendants of $\alpha \upharpoonright n$
labeled $\infty$ by $q$ are labeled OUT by $q_{alt}$. Observe
first that $q_{alt}$ extends $p$. Then note that, because
$proj_Q(q_{alt}) = proj_Q(q)$, the strong forcing facts posited of
$q$ hold for $q_{alt}$ as well: for the same $\tau$ and $j$ as
above, $q_{alt} \Vdash \tau \in proj_Q^{-1\prime\prime}\sigma$ and
$q_{alt} \Vdash ``\tau$ is an ordered pair with first component
$\hat j$," and if $q$ forced $\tau$'s second component to be
non-empty, $q_{alt}$ also forces it to be non-empty, containing
the same term as for $q$. The difference between $q$ and
$q_{alt}$, from $\sigma$'s point of view, is that $q_{alt}$ has
more extensions than $q$: there are conditions extending $q_{alt}$
which bar the tree beneath $\alpha$, which is not so for $q$. That
means that it is possible for extensions of $q_{alt}$ to force
sets into $Q$-terms that no extension of $q$ could. In the case of
$proj_Q^{-1\prime\prime}\sigma(\hat k)$, though, such
opportunities are limited. That term is already forced by $p$ to
be a function with domain $k$; for each $j < k$ there is already a
fixed term forced to stand for $\langle j,
(proj_Q^{-1\prime\prime}\sigma(\hat k))(j) \rangle$; if that
function value at $j$ was forced by $q$ to be 1 then it must
retain a member and so is also forced by $q_{alt}$ to be 1. The
only change possible is that something formerly forced to be empty
(i.e. be 0) could now be forced by some extension to have an
element (i.e. be 1). Recall, though, the construction of $q$ on
level $k$: if $proj_Q^{-1\prime\prime}\sigma(\hat k)$ is ever
forced by some $r \leq q_{alt}$ to be some $\beta \not = \alpha$
by flipping some 0's to 1's, by $\alpha$'s distinguished 1 $r$
cannot label $\beta$ $\infty$. So $r$ can be extended so that all
extensions of $\beta$ of a certain length are labeled IN, forcing
$proj_Q^{-1\prime\prime}\sigma$ to hit $\sigma_B$. Of course, any
extension of $q_{alt}$ forcing $proj_Q^{-1\prime\prime}\sigma(\hat
k)$ to be $\alpha$ works the same way as such an $r$ does, since
$q_{alt}$ already labels $\alpha$ OUT. In either case we have an
extension of $p$ forcing $proj_Q^{-1\prime\prime}\sigma$ go
through $\sigma_B$.
\end {proof}

Even though we have just seen that $B$ is a bar relative to the
$Q$-paths, we will perhaps surprisingly have occasion to consider
weaker situations, where $B$ is larger and hence even easier to
hit. The case of interest is if we were to change some $\infty$'s
in $G$ to OUTs, thereby allowing uniform bars above those nodes.
Notice that if $\alpha$'s sibling is not labeled $\infty$, then
$\alpha$'s label could not consistently be changed from $\infty$,
as then $\alpha$'s parent, labeled $\infty$, would then have both
children not labeled $\infty$. Such considerations do not apply
when $\alpha = \langle \rangle$.

\begin {Def} $H$ is a legal weakening of $G$ if $H$ can be
constructed by choosing finitely many nodes labeled $\infty$ by
$G$, changing those labels (to either IN or OUT), also changing
the labeling of finitely many descendants of those nodes from
$\infty$ or OUT to OUT or IN in such a way that each node labeled
OUT has a uniform bar above it labeled IN, and then eliminating
all descendants of nodes labeled IN. Furthermore, this must be
done in such a manner that $H$ is a filter through $P$ (avoiding,
for instance, the problem posed just before this definition).
\end {Def}

Notice that the difference between $H$ and $G$ can be summarized
in one condition $p$, which contains the new bars, all labeled IN,
and all of their ancestors. Hence we use the notation $G_p$ to
stand for this $H$: to build $G_p$, make the minimal change to
each condition in $G$ in order to be consistent with $p$.

\begin {Lem}
If $G_p$ is a legal weakening of $G$ then $G_p$ is generic through
$p$. \end {Lem}

\begin {Rmk} Notice that if $p$ labels the empty sequence IN or
OUT then $p = G_p$ is a terminal condition in $P$, trivially
satisfying the lemma.
\end {Rmk}

\begin {proof}
Let $D$ be dense beneath $p$. Notice that $G \upharpoonright
dom(p)$ is a condition in $P$ contained in $G$. It is not hard to
define the notion of projection beneath $p$, $proj_p$, by making
the minimal changes in a condition necessary to be compatible with
$p$. We claim that $proj_p^{-1\prime\prime}D$ is dense beneath $G
\upharpoonright dom(p)$. To see this, let $q \leq G
\upharpoonright dom(p)$. Extend $proj_p(q)$ to $r \in D$. The only
way $r$ can extend $proj_p(q)$ is by labeling extensions $\alpha$
of nodes which are unchanged by $proj_p$: if $\alpha \in dom(r)
\backslash dom(proj_p(q))$ then, for $\alpha \upharpoonright n \in
dom(q), \; q(\alpha \upharpoonright n) = proj_p(q)(\alpha
\upharpoonright n)$. Extend $q$ to $q_r$ by labeling those same
extensions the same way: for $\alpha \in dom(r) \backslash
dom(proj_p(q)) \; q_r(\alpha) = r(\alpha)$. We have that
$proj_p(q_r) = r$, hence $q_r \in proj_p^{-1\prime\prime}D$. So
$proj_p^{-1\prime\prime}D$ is dense beneath $G \upharpoonright
dom(p)$, hence contains a member of $G$, say $q$. Then $proj_p(q)$
is in both $D$ and $G_p$.
\end {proof}

We can now describe the ultimate Kripke model. Recall that $G$ is
generic for $P$ over $M$ and labels the empty sequence with
$\infty$. The bottom node $\bot$ of the Kripke model consists of
the $Q$-terms, with membership (not equality!) as interpreted by
$proj_Q(G)$. Let $N$ be an ultrapower of $M[G]$ using any
non-principal ultrafilter on $\omega$, with elementary embedding
$f:M[G] \rightarrow N$. This necessarily produces non-standard
integers. Let $\mathcal{H}$ be the set of legal weakenings of
$f(G)$, as defined in $N$, which induce the same $B$ on the
standard levels of 2$^*$, which restriction is definable only in
$M[G]$. That is, any standard node labeled $\infty$ by $G$ can
only be changed to OUT by the legal weakening. $\mathcal{H}$ will
index the successors of $\bot$. At the node indexed by $f(G)_p$,
the universe will be the $Q$-terms of $N$ as interpreted by
$proj_Q(f(G)_p)$. Regarding the embeddings from $\bot$, for a
$Q$-term $\sigma \in M$, $f(\sigma)$ is an $f(Q)$-term in $N$, so
send $\sigma$ to $f(\sigma)$. If $f(G)_p$ is a terminal condition
in $P$, then the node indexed by $f(G)_p$ is terminal in the
Kripke ordering. Else iterate. That is, suppose $f(G)_p$ is
non-terminal. The structure at its node can be built in $N$. As an
ultrapower of $M[G]$, $N$ internally looks like $f(M)[f(G)]$;
internally, $f(G)$ is $f(P)$-generic over the ground model $f(M)$.
The structure at node $f(G)_p$ could be built in $f(M)[f(G)_p]$,
where, by the previous lemma, $f(G)_p$ is generic through $f(P),$
and also non-terminal. Hence the construction just described,
using an ultrapower and legal weakenings to get additional nodes,
can be performed in $f(M)[f(G)_p]$ just as above. Continue through
$\omega$-many levels. We will ambiguously use $f$ to stand for any
of the elementary embeddings, including compositions of such
(making $f$ a sort-of polymorphic transition function). Notice
that the construction relativizes: the Kripke structure from node
$f(G)_p$ onwards is definable in $f(M)[f(G)_p]$ just as the entire
structure is definable in $M[G]$.

This defines a Kripke structure interpreting membership. Equality
at any node can now be defined as extensional equality beyond that
node in this structure, inductively on the ranks of the terms,
even though the model is not well-founded, thanks to the
elementarity present. That is, working at $\bot$, suppose $\sigma$
and $\tau$ are terms of rank at most $\alpha$, and we have defined
equality at $\bot$ for all terms of rank less than $\alpha$.
Moreover, suppose (strengthening the inductive assumption here)
that this definability was forced in $M$ by the empty condition
$\emptyset.$ At node $f(G)_p$ the structure is definable over
$f(M)[f(G)_p]$, and, by elementarity, in $f(M)$, $\emptyset \Vdash
``$Equality in the Kripke model is unambiguously definable for all
terms of rank less that $f(\alpha)$." So at that node we can see
whether there is a witness to $f(\sigma)$ and $f(\tau)$ being
unequal. If there is such a witness at any node $f(G)_p$, then
$\sigma$ and $\tau$ are unequal at $\bot$, else they are equal at
$\bot$. This extends the definability of equality to all terms of
rank $\alpha$. Hence inductively equality is definable for all
terms.

\begin {Pro}
$\bot \not \Vdash \FAND.$ \end {Pro}

\begin {proof}
It is immediate that $B$ is a bar: any node is internally of the
form $f(M)[f(G)_p]$; by the lemma, $f(G)_p$ is always
$f(P)$-generic; by the proposition, no path given by a $Q$-term
can avoid the interpretation of the term for $B$ as given by an
$f(P)$-generic. Moreover, $B$ is decidable, as $f(G)_p$ agrees
with $G$ on the standard part of 2$^*$, the only part that exists
at $\bot$, and that argument relativizes to all nodes. However,
$B$ is not uniform at any non-terminal node, since $f(G)_p$, when
non-terminal, has labels of $\infty$ at every level. \end {proof}

What remains to show is that our model satisfies IZF. In order to
do this, we will need to get a handle on internal truth in the
model. This is actually unnecessary for most of the IZF axioms,
but for Separation in particular we will have to deal with truth
in the model. When forcing, this is done via the forcing and truth
lemmas: $M[G] \models \phi$ iff $p \Vdash \phi$ for some $p \in
G,$ where $\Vdash$ is definable in $M$. Since our Kripke model is
built in $M[G]$, statements about it are statements within $M[G]$,
and so are forced by conditions in $G$. The problem is that the
Kripke model internally does not have access to $G$, but only to
$B$. In detail, Separation for $M[G]$ is proven as follows: given
$\phi$ and $\sigma$, it suffices to consider $\{ \langle q, \tau
\rangle \mid$ for some $\langle p, \tau \rangle \in \sigma, \; q
\leq p$ and $q \Vdash \phi(\tau) \}$. The problem we face is that
that set seems not to be in the Kripke model, even if $\sigma$ is.
What we need to show is that if $\sigma$ and $\phi$'s parameters
are $Q$-terms then that separating set is given by a $Q$-term.

Recall that proj$_Q$ operates by replacing all occurrences of
$\infty$ by OUT.

\begin {Def} $p \sim p'$ if proj$_Q(p)$ = proj$_Q(p')$.
\end {Def}

\begin {Def} $p \Vdash^* \phi$, for $\phi$ in the language of the
Kripke model, i.e. when $\phi$'s parameters are $Q$-terms,
inductively on $\phi:$
\begin {itemize}
\item $p \Vdash^* \sigma \in \tau$ if, for some
$\langle q, \rho \rangle \in \tau, \; q \geq_Q proj_Q(p)$ and $p
\Vdash^* \sigma = \rho.$
\item $p \Vdash^* \sigma = \tau$ if for all $p' \sim p$, $p'' \leq_P p'$, and
$\langle q, \rho \rangle \in \sigma,$ if $proj_Q(p'') \leq_Q q$
then there is a $p''' \leq_P p''$ such that $p''' \Vdash^* \rho
\in \tau$, and symmetrically.
\item $p \Vdash^* \phi \wedge \theta$ if $p \Vdash^* \phi$ and $p \Vdash^*
\theta$.
\item $p \Vdash^* \phi \vee \theta$ if $p \Vdash^* \phi$ or $p \Vdash^*
\theta$.
\item $p \Vdash^* \phi \rightarrow \theta$ if for all for all $p' \sim p$ and
$p'' \leq_P p'$,  if $p'' \Vdash^* \phi$ then there is a $p''' \leq_P p''$ such
that $p''' \Vdash^* \theta.$
\item $p \Vdash^* \exists x \; \phi(x)$ if, for some $Q$-term
$\sigma, \; p \Vdash^* \phi(\sigma)$.
\item $p \Vdash^* \forall x \; \phi(x)$ if for all for all $p' \sim p$, $p'' \leq_P p'$,
and $Q$-term $\sigma$, there is a $p''' \leq_P p''$ such that
$p''' \Vdash^* \phi(\sigma).$
\end {itemize}
\end {Def}
\begin {Lem} \label {equivalence}
If $p \sim p'$ then $p \Vdash^* \phi$ iff $p' \Vdash^* \phi.$
\end {Lem}

\begin {proof}
For the cases $\in, =, \rightarrow,$ and $\forall$, that is built
right into the definition of $\Vdash^*$. The other cases are a
trivial induction.
\end {proof}

\begin {Lem} \label {extension}
If $q \leq_P p$ and $p \Vdash^* \phi$ then $q \Vdash^* \phi.$
\end {Lem}

\begin {proof}
Inductively on $\phi$. For $\in$, use that $proj_Q$ is monotone.
The cases $\wedge, \vee,$ and $\exists$ are trivial inductions.
For the remaining cases, suppose $q'' \leq_P q', q' \sim q,$ and
$q \leq_P p.$ Then $q'' \leq_P q' \upharpoonright dom(p) \sim p$,
and use that $p \Vdash^* \phi$.
\end {proof}

\begin {Pro} \label {truth}
$\bot \models \phi$ iff $p \Vdash^* \phi$ for some $p \in G.$
\end {Pro}
\begin {proof}
Inductively on $\phi.$

$\sigma \in \tau$: $\bot \models \sigma \in \tau$ iff there are $p
\in G$ and $\langle q, \rho \rangle \in \tau$ such that $proj_Q(p)
\leq_Q q$ and $\bot \models \sigma = \rho$. Inductively $\bot
\models \sigma = \rho$ iff there is an $r \in G$ $*$-forcing the
same. In one direction, using lemma \ref {extension}, $p \cup r$
suffices, in the other we have $p = r$.

$\sigma = \tau$: Suppose $p \in G$ and $p \Vdash^* \sigma = \tau$.
By taking $p'$ equal to $p$ in the definition of $\Vdash^*$, for
every member $\rho$ of either $\sigma$ or $\tau$, it is dense to
$*$-force $\rho$ to be in the other set. By the genericity of $G$
some such $p'''$ will be in $G$, and so inductively $\rho$ will
end up in the other set. This shows that $\sigma$ and $\tau$ have
the same members at $\bot$. Regarding a future node $f(G)_{p''}$,
because $f(G)_{p''}$ is a legal weakening of $f(G)$, $p''
\upharpoonright dom(p) \sim p$, so again it is dense for any
member of $\sigma$ or $\tau$ to be forced into the other, so they
have the same members at node $f(G)_{p''}$. Hence $\bot \Vdash^*
\sigma = \tau.$

Conversely, suppose for all $p \in G \; p \not \Vdash^* \sigma =
\tau.$ That means there are $p' \sim p, p'' \leq_P p',$ and $\rho$
forced by $p''$ into $\sigma$ (without loss of generality), but
$p''$ has no extension $*$-forcing $\rho$ into $\tau$. For every
natural number $n$ the set $D_n = \{q \mid$ for some $k
> n, \; dom(q) \subseteq 2^k,$ and all binary sequences of length $k$
either are labeled $\infty$ by $q$ or some initial segment is
labeled IN by $q$\} is dense. Hence cofinally many levels of $G$
are in $D_0$. Observe that if $q$ is in $D_0 \cap G$ and $q' \sim
q$ then any extension of $q'$ can be extended again to induce a
legal weakening of $G$. In $N$, by overspill choose $p \in f(G)$
to be in $f(D_0)$. Choose $p'' \leq_P p' \sim p$ and $\rho$ as
given by the case hypothesis. Extend $p''$ to $p'''$ so that
$f(G)_{p'''}$ is a legal weakening of $f(G)$. Since $p'''$ has no
extension $*$-forcing $\rho$ into $\tau$, inductively at node
$f(G)_{p'''}$ $\rho$ is not a member of $\tau$. Hence $\bot \not
\models \sigma = \tau.$

$\phi \wedge \theta$: Trivial.

$\phi \vee \theta$: Trivial.

$\phi \rightarrow \theta$: Suppose $p \in G$ and $p \Vdash^* \phi
\rightarrow \theta$. At any node $f(G)_{p'},$ if $f(G)_{p'}
\models \phi$ then inductively choose $p'' \in f(G)_{p'}$ such
that $p'' \Vdash^* \phi.$ Without loss of generality $p''$ can be
taken to extend $p'$. Since $f(G)_{p'}$ indexes a node in the
model, $p' \upharpoonright dom(p) \sim p$, so $p'' \leq_P p'
\upharpoonright dom(p) \sim p$. By the case assumption there is a
$p'''$ extending $p''$ with $p''' \Vdash^* \theta.$ By the
genericity of $f(G)_{p'}$ there is such a $p'''$ in $f(G)_{p'}$.
So inductively $f(G)_{p'} \models \theta.$ At node $\bot$ the
argument is even simpler, as $p'$ can be chosen to be $p$. So
$\bot \models \phi \rightarrow \theta.$

Conversely, suppose for all $p \in G$ that $p \not \Vdash^* \phi
\rightarrow \theta.$ That means there are $p' \sim p$ and $p''
\leq_P p'$ with $p'' \Vdash^* \phi$ but no extension of $p'' \;
*$-forces $\theta$. As in the = case above, in $N$, by overspill
choose $p \in f(G)$ to be in $f(D_0)$. Choose $p'' \leq_P p' \sim
p$ as given by the case hypothesis. Extend $p''$ to $p'''$ so that
$f(G)_{p'''}$ is a legal weakening of $f(G)$. Inductively
$f(G)_{p'''} \models \phi,$ but since $p'''$ has no extension
$*$-forcing $\theta$, inductively $f(G)_{p'''} \not \models
\theta$. Hence $\bot \not \models \phi \rightarrow \theta.$

$\exists x \; \phi(x)$: Trivial.

$\forall x \; \phi(x)$: Suppose $p \in G$ and $p \Vdash^* \forall
x \; \phi(x)$. For any node $f(G)_{p'}$ and any $\sigma$ in the
universe there, $p' \leq_P p' \upharpoonright dom(p) \sim p$, so
there is a $p'' \leq_P p'$ such that $p'' \Vdash^* \phi(\sigma).$
By genericity there is such a $p''$ in $f(G)_{p'}$. Inductively
$f(G)_{p'} \models \phi(\sigma).$ So every element at node
$f(G)_{p'}$ satisfies $\phi$ there. At node $\bot$ the argument is
even easier, since $p'$ can be chosen to be $p$. Hence $\bot
\models \forall x \; \phi(x).$

Conversely, suppose for all $p \in G$ that $p \not \Vdash^*
\forall x \; \phi(x).$ That means there are $p' \sim p, p'' \leq_P
p',$ and $Q$-term $\sigma$ such that $p''$ has no extension
$*$-forcing $\phi(\sigma)$. As in the cases of = and $\rightarrow$
above, in $N$, by overspill choose $p \in f(G)$ to be in $f(D_0)$.
Choose $p'' \leq_P p' \sim p$ and $\sigma$ as given by the case
hypothesis. Extend $p''$ to $p'''$ so that $f(G)_{p'''}$ is a
legal weakening of $f(G)$. Since $p'''$ has no extension
$*$-forcing $\phi(\sigma)$, inductively $f(G)_{p'''} \not \models
\phi(\sigma)$. Hence $\bot \not \models \forall x \; \phi(x).$

\end {proof}

\begin {Thm} $\bot \models IZF$
\end {Thm}

\begin {proof}
Empty Set and Infinity are witnessed by $\emptyset$ and $\hat
\omega$ respectively. Pairing is witnessed by $\{ \langle
\emptyset, \sigma \rangle, \langle \emptyset, \tau \rangle \}$,
and Union by $\{ \langle q \cup r, \rho \rangle \mid$ for some
$\tau \; \langle q, \tau \rangle \in \sigma$ and $\langle r, \rho
\rangle \in \tau \}$. Extensionality holds because that's how =
was defined.

For $\epsilon$-Induction, suppose $\bot \models ``(\forall y \in x
\; \phi(y)) \rightarrow \phi(x)."$ If it were not the case that
$\bot \models ``\forall x \; \phi(x)"$, then at some later node
$G_p$ there would be a term $\sigma$ with $f(G)_p \not \models
\phi(\sigma)$. The restricted Kripke model of node $f(G)_p$ and
its extensions is definable in a model of ZF, say $N$, which is a
finite iteration of the ultrapower construction, and so is itself
a model of ZF. Hence, in $N,$ $\sigma$ can be chosen to be such a
term of least $V$-rank, say $\kappa$. Then at all nodes after
$f(G)_p$, by elementarity, it holds that $f(\kappa)$ is the least
rank of any term not satisfying $\phi$. So all members of
$\sigma$, being of lower rank, satisfy $\phi$ at whatever node
they appear. By the induction hypothesis, $\sigma$ must also
satisfy $\phi$, contradicting the assumption that some term does
not satisfy $\phi.$

For the powerset of $\sigma$ take all sets with members of the
form $\langle q, \tau \rangle$, where $\langle p, \tau \rangle \in
\sigma$ and $q \leq_Q p$.

It is easy to give a coarse proof of Bounding. The Kripke model
can be built in $M[G]$. Given a $\sigma$ at $\bot$, Bounding in
$M[G]$ can be used to bound the range of $\phi$ on $\sigma$ at
$\bot$. Also, the set of nodes is set-sized, so there are only
set-many interpretations of $f(\sigma)$ at the other nodes, so the
range of $\phi$ on them can also be bounded. Since the standard
ordinals are cofinal through the ordinals in all of the iterated
ultrapowers, by picking $\kappa$ large enough, $\hat {V_\kappa}$
suffices for bounding the range of $\phi$ on $\sigma$.

For Separation, given $\phi$ and $\sigma$, let Sep$_{\phi,
\sigma}$ be $\{ \langle proj_Q(p), \tau \rangle \mid $ for some
$\langle q, \tau \rangle \in \sigma$ with $p \leq q$ we have $p
\Vdash^* \phi(\sigma) \}$. By lemmas \ref {equivalence} and \ref
{truth}, this works.
\end {proof}

Although this model does not satisfy \FAND, it does satisfy
$\neg\neg$\FAND, as the terminal nodes are dense. Admittedly this
is a rather weak failure of \FAND. In the final section, we will
address the issue of getting stronger failures of \FAND.

\section {\FAND does not imply \FANc}
We will need a tree similar to that of the last proof. In fact, we
will need two trees: the $c-$bar $C$, and the decidable set $C'$
from which $C$ is defined. (Both can be viewed either as 2$^*$
with labels or as subtrees of 2$^*$.) Mostly we will focus on $C$.
Because \FANc refers to eventual membership in a tree, the
difference between IN and OUT nodes is no longer relevant: the bar
is uniform beneath any OUT node. So we can describe the forcing in
terms similar to those before, and with some simplifications
introduced. The forcing partial order {\it P} will be the set of
appropriate labelings of finitely many nodes from 2$^*$. A
labeling of nodes assigns to each one either IN or $\infty$, with
the following restrictions. Any node labeled IN has no descendant,
the idea being that once a node gets into the eventual bar so are
all of its descendants automatically, so nothing more need be
said. For any node labeled $\infty$, if both children are labeled,
then at least one of them must be labeled $\infty$. Let $G$ be
{\it P}-generic through the condition labeling the empty sequence
with $\infty$.

As before, we will need to look at weaker trees, ones with bigger
bars.

\begin {Def} $H$ is a legal weakening of $G$ if $H$ can be
constructed by choosing finitely many nodes labeled $\infty$ by
$G$, whose siblings are also labeled $\infty$ by $G$, and changing
those labels to IN and eliminating all descendants.
\end {Def}

As before, each legal weakening $H$ can be summarized by one
forcing condition $p$, which consists of those nodes changed by
$H$ and their ancestors, labeled as in $G$. $H$ is then the set of
conditions in $G$ each minimally changed to be consistent with
$p$. Hence we refer to $H$ as $G_p$.

\begin {Lem}
If $G_p$ is a legal weakening of $G$ then $G_p$ is generic through
$p$.
\end {Lem}

\begin {proof}
As in the corresponding lemma in the previous section.
\end {proof}

\begin {Def} Terms are defined inductively (through the
ordinals) as sets of the form $\{ \langle B_i, \sigma_i \rangle
\mid i \in I\}$, where $I$ is any index set, $\sigma_i$ a term,
and $B_i$ a finite set of truth values. A truth value is a symbol
of the form $b^+$ or $b'$ or $\neg b'$, for $b \in 2^*$ a finite
binary sequence.
\end {Def}

\begin {Def} Let $C$ be the term $\{ \langle \{b^+\}, \hat{b}
\rangle \mid b \in 2^*\}$, and $C'$ be $\{ \langle \{b'\}, \hat{b}
\rangle \mid b \in 2^*\}$.
\end {Def}

In our final model, (the interpretation of) $C$ will be the
$c$-bar induced by (the interpretation of) $C'$, and $C$ will not
be uniform, thereby falsifying \FANc. Furthermore, we will show
that \FAND holds in this model.

We can now describe the ultimate Kripke model. Recall that $G$ is
generic for $P$ over $M$ and labels the empty sequence with
$\infty$. The bottom node $\bot$ of the Kripke model consists of
the terms. At $\bot$, $b^+$ counts as true iff $G(b)$ = IN, $b'$
always counts as true, and $\neg b'$ never counts as true. Later
nodes will have different ways of counting the various literals as
true. At any node, for $\sigma = \{ \langle B_i, \sigma_i \rangle
\mid i \in I\}$, if each member of some $B_i$ counts as true, then
at that node $\sigma_i \in \sigma$. This induces a notion of
extensional equality among the terms. One way of viewing this is
at any node to remove from a term $\sigma$ any pair $\langle B_i,
\sigma_i \rangle$ if some member of $B_i$ is not true at that
node. Then each remaining $\langle B_i, \sigma_i \rangle$ can be
replaced by $\sigma_i$. Equality is then as given by the Axiom of
Extensionality as interpreted in the model.

As for what the other nodes in the model are, there are two
different kinds. As in the last section, let $N$ be an ultrapower
of $M[G]$ using any non-principal ultrafilter on $\omega$, with
elementary embedding $f:M[G] \rightarrow N$. This necessarily
produces non-standard integers. In $N$, any forcing condition $p$
which induces a legal weakening of $f(G)$ will index a successor
node to $\bot$. At the node indexed by $p$, the universe will be
the terms of $N$ as interpreted by $f(G)_p$. That is, $b^+$ is
true if $f(G)_p(b)$ = IN, $b'$ is always true, and $\neg b'$
never. Regarding the embeddings from $\bot$, for a term $\sigma
\in M$, $f(\sigma)$ is a term in $N$, so send $\sigma$ to
$f(\sigma)$. In addition, definably over $M[G]$, any non-standard
$c \in 2^*$ with $f(G)(c) = \infty$ also indexes a node. At such a
node $c$, $b'$ counts as true iff $b \not = c$, $\neg b'$ counts
as true iff $b = c$, and $b^+$ counts as true iff $b \not
\subseteq c$ ($b$ is not an initial segment of $c$). Note that at
$\bot$ any $b'$ refers only to a standard $b$; for some $b'$ to be
declared false at a later node $c$, $b$ would have to equal $c$,
and $c$ indexes a node only if $c$ is non-standard. Hence there is
no conflict with the Kripke structure: once $b'$ is deemed true,
it remains true. Similarly with $b^+$: $G_p$ is a fattening of
$G$. Hence membership, being based on finitely many truth values,
is monotone.

Any node indexed by such a $c \in 2^*$ is terminal in the Kripke
ordering. Also, among nodes of the other kind, there is one
trivial condition $p$, the one with $p(\langle\rangle)$ = IN. This
is also a terminal node, where each $b^+$ and each $b'$ is true.
At any other node, iterate. That is, suppose $p$ is not the
preceding condition. The model at $p$ can be built in $N$. As an
ultrapower of $M[G]$, $N$ internally looks like $f(M)[f(G)]$. The
structure at node $p$ could be built in $f(M)[f(G)_p]$, where
$f(G)_p$ is generic through $f(P)$ (and non-trivial). Hence the
construction just described, using an ultrapower and legal
weakenings and non-standard binary strings to get additional
nodes, can be performed in $f(M)[f(G)_p]$ just as above. This
provides immediate successors to nodes indexed by (non-trivial)
$p$'s. Iterate $\omega$-many times.

The picture is that, at $\bot,$ $C$ looks like $G$, that is, those
nodes $G$ assigns to be IN. This tree gets fatter at later nodes
that are legal weakenings. At terminal nodes $c$, $C$ is
everything but the branch up to $c$. At most nodes $C'$ looks like
everything; at node $c$, where $c$ is non-standard relative to its
predecessor, we find the one thing not in $C'$, namely $c$.

What we need to show is that this model satisfies IZF and \FAND,
and falsifies \FANc.

\begin {Lem} $\bot \not \models$ \FANc.

\end {Lem}

\begin {proof}
It is easy to see that $C$ is the $c-$set induced by $C'$: once
$b$ is forced into $C$, none of its descendants index terminal
nodes, so no descendant is forced out of $C'$; similarly, if $b$
is not forced into $C$, say at node $p$, then $G_p(b) = \infty$,
and in $N$ some non-standard extension $c$ of $b$ will also be
labeled $\infty$ by $f(G)_p$, and that $c$ will index a node at
which $c$ is not in $C'$. Clearly, $C$ is not uniform, and $C'$ is
decidable. So it remains only to show that $\bot \models C$ is a
bar.

Suppose $\sigma$ is forced to be an infinite binary path at some
node. If that node is a terminal node, $C$ contains cofinitely
many members of $2^*$, and so certainly intersects $\sigma$. Else
without loss of generality we can assume the node is $\bot$. Then,
for some $p \in G, \; p \Vdash ``\bot \models \sigma$ is an
infinite binary path." If it is not dense beneath $p$ to force the
standard part of $\sigma$ (that is, $\sigma$ applied to the
standard integers) to be in the ground model, then extensions $q$
and $r$ of $p$ force incompatible facts about $\sigma$. The only
incompatible facts about $\sigma$ are of the form $b^\frown0 \in
\sigma$ and $b^\frown1 \in \sigma$. The positive parts of $q$ and
$r$ (that is, $q^{-1}(IN)$ and $r^{-1}(IN)$) induce a legal
weakening of $G$. That is, there is a canonical condition
inpart($q,r$), with domain $dom(q) \cup dom(r)$, that returns IN
on any node that either $q$ or $r$ returns IN on, as well as on
any node if inpart($q,r$) returns IN on both children, else OUT.
Because terms use only positive (i.e. IN) information, at the node
$f(G)_{inpart(q,r)}$, both $b^\frown0$ and $b^\frown1$ are in
$\sigma$. (More coarsely and perhaps more simply, at the node
induced by the trivial condition sending the empty sequence to IN,
the same conclusion holds for the same reason.) Hence $\bot$ could
not have forced $\sigma$ to be a path in the first place.
Therefore $p$ forces $\sigma$ on the standard binary tree to be in
the ground model. It is easy to see that generically $G$ labels
some node in $\sigma$ IN.
\end {proof}

\begin {Lem}
$\bot \models \FAND$
\end {Lem}

\begin {proof}
By arguments similar to the above. If a set of nodes $B$ is forced
by $p$ to be decidable, then no extensions of $p$ can force
incompatible facts about $B$. Hence $B$ is in the ground model. If
$B$ were not a bar in the ground model, there would be a ground
model path missing $B$. This path would also be in the Kripke
model. Hence $B$ is a bar in the ground model. Since the ground
model is taken to be classical, $B$ is uniform.
\end {proof}

Regarding getting IZF to be true, just as in the previous section,
the problem is that truth in the Kripke model is on the surface
determined by forcing conditions in the ground model, to which the
Kripke model has no access. The essence is to capture truth at a
node using those truth values that are allowed in the building of
terms.

\begin {Def} For a forcing condition $p, \; B_p = \{b^+ \mid$ for
some initial segment $c$ of $b, \; p(c)$ = IN\}.

For a set of truth values $B, \; B^+ = B \cap \{b^+ \mid b \in
2^*\}$. Also, $B$ is positive if $B$ contains no truth value of
the form $\neg b'$.
\end {Def}

\begin {Def} \begin {enumerate}
\item $\neg b' \Vdash^* B$ iff $c^+ \in B \rightarrow c^+ \not
\subseteq b', \; c' \in B \rightarrow c \not = b,$ and $\neg c'
\in B \rightarrow c = b.$
\item $\sigma^{\neg b'} = \{\sigma_i^{\neg b'} \mid$ for some
$\langle B_i, \sigma_i \rangle \in \sigma, \; \neg b' \Vdash^*
B_i\}$.
\item For $\phi(\sigma_1, \ldots, \sigma_n)$ in the language of the
Kripke model, that is, with parameters (displayed) terms,
$\phi^{\neg b'} = \phi(\sigma_1^{\neg b'}, \ldots, \sigma_n^{\neg
b'})$.
\item $\neg b' \Vdash^* \phi$, for $\phi$ in the language
of the Kripke model, if $\phi^{\neg b'}$ is true (i.e. in $V$).
Note that $\phi^{\neg b'}$ is a formula with set parameters.
\end {enumerate}

\end {Def}

\begin {Def}
$q \leq_W p$ ($q$ is a weakening of $p$ as conditions) if for $b
\in dom(p)$ either $p(b) = \infty$ or for some initial segment $c$
of $b \; q(c)$ = IN.
\end {Def}

The idea behind this definition is the $q$ may change some
$\infty$'s to IN's, as well as extend the domain of $p$. Notice
that $\leq_W$ is a partial order, and inpart($p,q$), from lemma
14, is the greatest lower bound of $p$ and $q$.

\begin {Def} $p \Vdash^* \phi$, for $\phi$ in the language of
the Kripke model, i.e. when $\phi$'s parameters are terms,
inductively on $\phi:$
\begin {itemize}
\item $p \Vdash^* \sigma \in \tau$ if for some
$\langle B_i, \tau_i \rangle \in \tau$ with $B_i$ positive, $B_i^+
\subseteq B_p$ and $p \Vdash^* \sigma = \tau_i.$
\item $p \Vdash^* \sigma = \tau$ if\\
$i)$ for all $\langle B_i, \sigma_i \rangle \in \sigma$ and $q
\leq_W p,$ if $B_i$ is positive and $B_i^+ \subseteq B_q$ then
there is an $r \leq q$ such that $r \Vdash^* \sigma_i \in \tau$,
and symmetrically between $\sigma$ and $\tau$, and\\ $ii)$ for all
$b \not \in dom(p)$, if for no initial segment $c$ of $b$ is $c^+$
in $B_p$, then $\neg b' \Vdash^* \sigma = \tau$.
\item $p \Vdash^* \phi \wedge \theta$ if $p \Vdash^* \phi$ and $p
\Vdash^* \theta$.
\item $p \Vdash^* \phi \vee \theta$ if $p \Vdash^* \phi$ or $p
\Vdash^* \theta$.
\item $p \Vdash^* \phi \rightarrow \theta$ if\\ $i)$ for all
$q \leq_W p$ if $q \Vdash^* \phi$ then there is an $r \leq q$ such
that $r \Vdash^* \theta,$ and\\ $ii)$ for all $b \not \in dom(p)$,
if for no initial segment $c$ of $b$ is $c^+$ in $B_p$, then $\neg
b' \Vdash^* \phi \rightarrow \theta$.
\item $p \Vdash^* \exists x \; \phi(x)$ if for some term
$\sigma \; p \Vdash^* \phi(\sigma)$.
\item $p \Vdash^* \forall x \; \phi(x)$ if\\ $i)$ for all terms $\sigma$
and $q \leq_W p,$ there is an $r \leq q$ such that $r \Vdash^*
\phi(\sigma),$ and\\ $ii)$ for all $b \not \in dom(p)$, if for no
initial segment $c$ of $b$ is $c^+$ in $B_p$, then $\neg b'
\Vdash^* \forall x \; \phi(x)$.
\end {itemize}
\end {Def}

\begin {Pro} If $p \Vdash^* \phi$ and $q \leq_W p$ then $q \Vdash^* \phi.$
\end {Pro}
\begin {proof}
Trivial induction on $\phi.$
\end {proof}

\begin {Lem}
$\bot \models \phi$ iff $p \Vdash^* \phi$ for some $p \in G.$
\end {Lem}

\begin {proof}
Inductively on $\phi.$

$\sigma \in \tau$: $\bot \models \sigma \in \tau$ iff for some
$\langle B_i, \tau_i \rangle \in \tau$ every member of $B_i$ is
true at $\bot$ and $\bot \models \sigma = \tau_i$. The former
clause holds iff $B_i$ is positive and, for some $p \in G, \; B_i
\subseteq B_p$. Inductively, the latter clause holds iff, for some
$q \in G, \; B_q \Vdash^* \sigma = \tau_i.$ Given such $p$ and
$q$, $p \cup q$ suffices. The converse direction is immediate.

$\sigma = \tau$: Suppose $p \in G$ and $p \Vdash^* \sigma = \tau.$
If $q \in f(P)$ indexes a node then $q \leq_W p$. If $q \models
\rho \in \sigma$ then inductively there is a $q' \in f(G)_q, \; q'
\leq q,$ such that $q' \Vdash^* \rho = \sigma_i \wedge \sigma_i
\in f(\sigma)$ for some $\langle B_i, \sigma_i \rangle \in
f(\sigma).$ By $i)$ of the case hypothesis, there is an $r \leq
q'$ with $r \Vdash^* \rho \in \tau.$ Generically, there is such an
$r$ in $f(G)$, so inductively $q \models \rho \in \tau.$ If $c$
indexes a node, then by $ii)$ of the case hypothesis $c \models
\sigma = \tau$. Hence $\bot \models \sigma = \tau.$

Conversely, suppose there is no such $p \in G.$ If $p \not
\Vdash^* \sigma = \tau$ because clause $i)$ fails, then there is a
witness $q \leq_W p$ to that failure. We say that such a $q$ is
close to $p$ if $dom(q) \subseteq dom(p)$. That means that $q$
comes from $p$ by changing some $\infty$'s to IN's and not adding
anything else. Observe that if $i)$ fails for $p$, then $p$ can be
extended to $p'$ so that $i)$ fails for $p'$ via a witness $q$
close to $p'$. That's because $dom(p')$ can be taken to be $dom(p)
\cup dom(q)$, for $b \in dom(p) \; p'(b)$ can be taken to be
$p(b)$, and for $b \in dom(q) \backslash dom(p) \; p'(b)$ can be
taken to be $q(b)$. Therefore $D = \{p \mid p \Vdash^* \sigma =
\tau,$ or $p$ violates $i)$ with a witness $q$ close to $p$, or
$p$ violates $ii)$\} is dense.

Suppose there were a $p \in G$ violating $i)$ with a witness $q$
close to $p$. Then $q$ induces a legal weakening $f(G)_q$ of
$f(G)$, and so indexes a node. By the choice of $q, \; q \models
\sigma_i \in \sigma.$ If $q \models \sigma_i \in \tau$ then
inductively that would be $*$-forced by some $r \leq q$. But by
the choice of $q$ there is no such $r$. Hence we would have $q
\not \models \sigma = \tau.$

If there is no such $p$ then every $p \in G$ violates $ii)$. Let
$p \in f(G)$ be such that $p \supseteq G$. Since $ii)$ fails for
that $p$, then, with $b$ from that failure, $b$ indexes a node, $b
\not \models \sigma = \tau$. In either case, $\bot \not \models
\sigma = \tau.$

$\phi \wedge \theta$: Trivial.

$\phi \vee \theta$: Trivial.

$\phi \rightarrow \theta$: Suppose $p \in G$ and $p \Vdash^* \phi
\rightarrow \theta$. If $q \models \phi$ then inductively, for
some $q' \in f(G)_q, \; q' \Vdash^* \phi$. Since we can take $q'
\leq q \leq_W p$, by $i)$ of the hypothesis there is an $r \leq q$
such that $r \Vdash^* \theta.$ By genericity, there is such an $r$
in $f(G)_q$. Hence $q \models \theta.$ If $c \models \phi$ then
use $ii)$ of the hypothesis.

Conversely, suppose there is no such $p \in G.$ If $p \not
\Vdash^* \phi \rightarrow \theta$ because clause $i)$ fails, then
there is a witness $q \leq_W p$ to that failure, in which case $p$
can be extended to $p'$ so that $i)$ fails for $p'$ via a witness
$q$ close to $p'$, where closeness is as defined above in the case
for =, for the same reason as above. Therefore $D = \{p \mid p
\Vdash^* \phi \rightarrow \theta,$ or $p$ violates $i)$ with a
witness $q$ close to $p$, or $p$ violates $ii)$\} is dense.

Suppose there were a $p \in G$ violating $i)$ with a witness $q$
close to $p$. Then $q$ induces a legal weakening $f(G)_q$ of
$f(G)$, and so indexes a node. By the choice of $q, \; q \models
\phi.$ If $q \models \theta$ then inductively that would be
$*$-forced by some $r \leq q$. But by the choice of $q$ there is
no such $r$. Hence we would have $q \not \models \phi \rightarrow
\theta.$

If there is no such $p$ then every $p \in G$ violates $ii)$. Let
$p \in f(G)$ be such that $p \supseteq G$. Since $ii)$ fails for
that $p$, then, with $b$ from that failure, $b$ indexes a node and
$b \not \models \phi \rightarrow \theta$. In either case, $\bot
\not \models \phi \rightarrow \theta.$

$\exists x \; \phi(x)$: Trivial.

$\forall x \; \phi(x)$: As in the cases for = and $\rightarrow.$

\end {proof}

\begin {Lem}
$\bot \models IZF$
\end {Lem}

\begin {proof}
Just as in the last section, most of the axioms have soft proofs
in this model. The only issue is with Separation. Given $\phi$ and
$\sigma$, let Sep$_{\phi, \sigma}$ be $\{ \langle B, \tau \rangle
\mid $ for some $\langle B', \tau \rangle \in \sigma$ with $B
\supseteq B'$ either $i) \; B = B_p \Vdash \phi(\sigma)$, or $ii)
\; \neg b' \in B$ and $\neg b' \Vdash^* \phi \}$. By the previous
lemma, this works.
\end {proof}

As in the previous section, this model does not satisfy \FANc, but
does satisfy $\neg \neg$ \FANc. For further discussion, see the
questions at the end.

\section {\FANc does not imply \FANP}
Let $G$ be $P$-generic exactly as in the last section. By
convention, we say that if $G(\alpha)$ = IN then $G$ applied to
any extension of $\alpha$ is also IN. Our goal is to hide $G$ a
bit better than before, so \FANc remains true, but not too well,
so that \FANP is false.

Let $N$ be an ultrapower of $M[G]$ using a non-principal
ultrafilter on $\omega$. The Kripke model has a bottom node
$\bot$, and the successors of $\bot$ are indexed by the labels
$\langle n, \alpha \rangle$, where $n$ is a non-standard integer,
and $\alpha \in 2^*$ either has non-standard length or $G(\alpha)
= \infty$.

\begin {Def} A truth value is a symbol of the form $\langle n,
\alpha \rangle$, $\neg \langle n, \alpha \rangle$, or $\langle
\forall n, \alpha \rangle$, for $n$ a natural number (in the first
two cases) and $\alpha \in 2^*$. Admittedly truth values of the
first kind are also used to index nodes; whether truth values or
nodes are intended in any particular case should be clear from the
context. Terms are defined inductively (through the ordinals) as
sets of the form $\{ \langle B_i, \sigma_i \rangle \mid i \in
I\}$, where $I$ is any index set, $\sigma_i$ a term, and $B_i$ a
finite set of truth values.
\end {Def}

The sets at $\bot$ will be the terms in $M$. The sets at any other
node will be analogous, that is, the terms in what $N$ thinks is
the ground model, i.e. $\bigcup_{\kappa \in ORD} f(M_\kappa)$. At
$\bot$, $\langle n, \alpha \rangle$ will always be true, $\neg
\langle n, \alpha \rangle$ always false, and $\langle \forall n,
\alpha \rangle$ true exactly when $G(\alpha) =$ IN. At node
$\langle m, \beta \rangle$, $\langle n, \alpha \rangle$ is true
exactly when $\langle n, \alpha \rangle \not = \langle m, \beta
\rangle$, $\neg \langle n, \alpha \rangle$ is true exactly when
$\langle n, \alpha \rangle = \langle m, \beta \rangle$, and
$\langle \forall n, \alpha \rangle$ true exactly when $\alpha \not
= \beta.$ (Note that, perhaps perversely, the node $\langle n,
\alpha \rangle$ is exactly the node at which the truth value
$\langle n, \alpha \rangle$ is {\it false}. The reason behind this
choice is that the node $\langle n, \alpha \rangle$ is where
something special happens to the corresponding truth value. If
preferred, the reader can call that node $\neg \langle n, \alpha
\rangle$ instead.) This interpretation of the truth values induces
an interpretation of the terms at all nodes.

Let $T_n$ be the term $\{ \langle \{\langle n, \alpha \rangle \},
\hat{\alpha} \rangle \mid \alpha \in 2^* \}$. Let $C$ be a term
naming the function that on input $n$ returns $T_n$. $T_n$ at
$\bot$ looks like the full tree $2^*$; $T_n$ at $\langle n, \alpha
\rangle$ looks like everything except $\alpha$; and $T_n$ at
$\langle m, \alpha \rangle, \; m \not = n,$ again looks like
$2^*.$ The term for $\bigcap_n C(n)$ is given by $\{ \langle
\{\langle \forall n, \alpha \rangle \}, \hat{\alpha} \rangle \mid
\alpha \in 2^* \},$ and is interpreted as $\{ \alpha \mid
G(\alpha) =$ IN \} at $\bot$ and $2^* \backslash \{\alpha\}$ at
$\langle n, \alpha \rangle.$ Notice that $\bigcap_n C(n)$ is not
closed under extensions.

The proof will be finished once we show that, at $\bot$, \FANc
holds, IZF holds, and $\bigcap_n C(n)$ is a counter-example to
\FANP.

\begin {Lem} $\bot \not \models$ \FANP.
\end {Lem}

\begin {proof}
It is clear that $T_n$ is decidable, and so $\bigcap_n C(n)$ is on
the face of it $\Pi^0_1.$ It is also clear that $\bigcap_n C(n)$
is not a uniform bar. So it suffices to show that $\bot \Vdash
``\bigcap_n C(n)$ is a bar."

Let $\bot \models ``Br$ is a branch through $2^*."$ (Without loss
of generality, it suffices to start at $\bot$ instead of at an
arbitrary node.) Work beneath a condition forcing that, so we can
assume $Br$ consists of sets of the form $\langle B_i,
\hat{\alpha} \rangle$, for various $\alpha \in 2^*$. If the
standard part of $Br$, the part visible at $\bot$, is in the
ground model $M$, then, by the genericity of $G, \; Br$ will hit
$G$ (i.e. for some $\alpha \in Br, \; G(\alpha)$ = IN), which is
how $\bot$ interprets $\bigcap_n C(n)$. If the standard part of
$Br$ were not in $M$, then contradictory facts about $Br$ would be
forced by different forcing conditions. In particular, we would
have $p, q,$ and $\alpha$ with $p \Vdash ``\bot \models
\alpha^\frown 0 \in Br"$ and $q \Vdash ``\bot \models
\alpha^\frown 1 \in Br."$ That means there are $\langle B_p,
\widehat{\alpha^\frown 0} \rangle \in Br$ and $\langle B_q,
\widehat{\alpha^\frown 1} \rangle \in Br,$ with $B_p$ and $B_q$
consisting only of truth values automatically true at $\bot$ save
for some of the form $\langle \forall n, \alpha \rangle$. But at
some node $\langle n, \alpha \rangle$ with $\alpha$ non-standard,
all of those latter truth values will be true. Hence $\langle n,
\alpha \rangle \models ``\widehat{\alpha^\frown 0},
\widehat{\alpha^\frown 1} \in Br,"$ so $\bot$ could not force $Br$
to be a path.
\end {proof}

In order to prove the other facts, we will need to deal with truth
at $\bot$.

\begin {Def} For a forcing condition $p$, let $\mid p \mid$, the
length of $p$, be the length of the longest $\alpha \in dom(p)$.
Let $B_p$ be $\{\langle n, \alpha \rangle \mid n,$ length$(\alpha)
\leq \mid p \mid\} \cup \{\langle \forall n, \alpha \rangle \mid$
length$(\alpha) \leq \mid p \mid$ and, for some initial segment
$\beta$ of $\alpha, \; p(\beta)$ = IN\}.
\end {Def}

\begin {Def} \begin {enumerate}
\item For $B$ a finite set of
truth values, $\neg \langle n, \alpha \rangle \Vdash^* B$ iff
$\langle n, \alpha \rangle \not \in B, \; \langle \forall n,
\alpha \rangle \not \in B, \;$ and the only truth value of the
form $\neg \langle m, \beta \rangle$ in $B$ is $\neg \langle n,
\alpha \rangle$ itself.
\item $\sigma^{\neg \langle n, \alpha \rangle} = \{\sigma_i^{\neg
\langle n, \alpha \rangle} \mid$ for some $\langle B_i, \sigma_i
\rangle \in \sigma, \; \neg \langle n, \alpha \rangle \Vdash^*
B_i\}$.
\item For $\phi(\sigma_1, \ldots, \sigma_n)$ in the language of the
Kripke model, that is, with parameters (displayed) terms,
$\phi^{\neg \langle n, \alpha \rangle} = \phi(\sigma_1^{\neg
\langle n, \alpha \rangle}, \ldots, \sigma_n^{\neg \langle n,
\alpha \rangle})$.
\item $\neg \langle n, \alpha \rangle \Vdash^* \phi$, for $\phi$
in the language of the Kripke model, if $\phi^{\neg \langle n,
\alpha \rangle}$ is true (i.e. in $V$). Note that $\phi^{\neg
\langle n, \alpha \rangle}$ is a formula with set parameters.
\end {enumerate}

\end {Def}

\begin {Def} $p \Vdash^* \phi$, for $\phi$ in the language of the Kripke model, i.e.
when $\phi$'s parameters are terms, inductively on $\phi:$
\begin {itemize}
\item $p \Vdash^* \sigma \in \tau$ if, for some
$\langle B_i, \tau_i \rangle \in \tau$, $B_i \subseteq B_p$ and $p
\Vdash^* \sigma = \tau_i.$
\item $p \Vdash^* \sigma = \tau$ if\\
$i)$ for all $\langle B_i, \sigma_i \rangle \in \sigma$ and $q
\leq p$, if $B_i \subseteq B_q$ then there is an $r \leq q$ such
that $r \Vdash^* \sigma_i \in \tau$, and symmetrically between
$\sigma$ and $\tau$, and\\ $ii)$ if $n > \mid p \mid$, and if
either length($\alpha) > \mid p \mid$ or for no initial segment
$\beta$ of $\alpha$ do we have $p(\beta)$ = IN, then $\neg \langle
n, \alpha \rangle \Vdash^* \sigma = \tau$.
\item $p \Vdash^* \phi \wedge \theta$ if $p \Vdash^* \phi$ and $p
\Vdash^* \theta$.
\item $p \Vdash^* \phi \vee \theta$ if $p \Vdash^* \phi$ or $p
\Vdash^* \theta$.
\item $p \Vdash^* \phi \rightarrow \theta$ if\\ $i)$ for all
$q \leq p$, if $q \Vdash^* \phi$ then there is an $r \leq q$ such
that $r \Vdash^* \theta,$ and\\ $ii)$ if $n > \mid p \mid$, and if
either length($\alpha) > \mid p \mid$ or for no initial segment
$\beta$ of $\alpha$ do we have $p(\beta)$ = IN, then $\neg \langle
n, \alpha \rangle \Vdash^* \phi \rightarrow \theta$.
\item $p \Vdash^* \exists x \; \phi(x)$ if for some term
$\sigma \; p \Vdash^* \phi(\sigma)$.
\item $p \Vdash^* \forall x \; \phi(x)$ if\\ $i)$ for all terms $\sigma$
and $q \leq p$, there is an $r \leq q$ such that $r \Vdash^*
\phi(\sigma),$ and\\ $ii)$ if $n > \mid p \mid$, and if either
length($\alpha) > \mid p \mid$ or for no initial segment $\beta$
of $\alpha$ do we have $p(\beta)$ = IN, then $\neg \langle n,
\alpha \rangle \Vdash^* \forall x \; \phi(x)$.
\end {itemize}
\end {Def}

\begin {Pro} If $p \Vdash^* \phi$ and $q \leq p$ then $q \Vdash^* \phi.$
\end {Pro}
\begin {proof}
Trivial induction on $\phi.$
\end {proof}

\begin {Lem} \label{truth2}
$\bot \models \phi$ iff $p \Vdash^* \phi$ for some $p \in G.$
\end {Lem}

\begin {proof}
Inductively on $\phi.$

$\sigma \in \tau$: $\bot \models \sigma \in \tau$ iff for some
$\langle B_i, \tau_i \rangle \in \tau$ every member of $B_i$ is
true at $\bot$ and $\bot \models \sigma = \tau_i$. The former
clause holds iff $B_i$ contains nothing of the form $\neg \langle
n, \alpha \rangle$, and if $\langle \forall n, \alpha \rangle \in
B_i$ then $G(\alpha)$ = IN. Given such a $B_i$, let $p$ be a
sufficiently long initial segment of $G$ forcing $``\sigma =
\tau_i$." Such a $p$ suffices. The converse direction is
immediate.

$\sigma = \tau$: Suppose $p \in G$ and $p \Vdash^* \sigma = \tau.$
Then any member of $\sigma$ at $\bot$ is equal at $\bot$ to some
$\sigma_i$, where $\langle B_i, \sigma_i \rangle \in \sigma$ and
$B_i \subseteq B_q$ for some $q \in G$. Then by the hypothesis and
genericity there will be an extension $r$ of $q$ in $G$ forcing
$\sigma_i$ to be in $\tau$. At any other node $\langle n, \alpha
\rangle$, working in $N$, $n$ is non-standard and so greater than
$\mid p \mid$, and $\alpha$ also satisfies the conditions in $ii)$
(of the definition of $*$-forcing equality), so $``\sigma = \tau"$
is true at these other nodes too.

Conversely, suppose there is no such $p \in G.$ With reference to
the definition of $*$-forcing equality, observe that $\{ p \mid p$
satisfies clause $i) \} \cup \{ p \mid$ for some $\langle B_i,
\sigma_i \rangle \in \sigma, \; B_i \subseteq B_p,$ yet $p$ has no
extension $*$-forcing $\sigma_i$ into $\tau \}$ is dense. If $G$
contains a member of the second set of that union, then the
induced $\sigma_i$ witnesses that $\bot \not \models \sigma =
\tau$. If not, then $G$ contains $p$ satisfying $i)$. So no $p \in
G$ satisfies $ii)$. This also holds in $N$. In $N$, take $p$ to be
an initial segment of $G$ of non-standard length. The failure of
$ii)$ for that $p$ produces an $n$ and $\alpha$ which index a node
at which $\sigma \not = \tau$, showing $\bot \not \models \sigma =
\tau.$

$\phi \wedge \theta$: Trivial.

$\phi \vee \theta$: Trivial.

$\phi \rightarrow \theta$: Suppose $p \in G$ and $p \Vdash^* \phi
\rightarrow \theta$. Then it is direct that $\bot \models \phi
\rightarrow \theta.$

Conversely, suppose there is no such $p \in G.$ With reference to
the definition of $*$-forcing implication, observe that $\{ p \mid
p$ satisfies clause $i) \} \cup \{ p \mid p \Vdash^* \phi$ yet $p$
has no extension $*$-forcing $\theta \}$ is dense. If $G$ contains
a member of the second set of that union, then inductively $\bot
\models \phi$ and $\bot \not \models \theta$, hence $\bot \not
\models \phi \rightarrow \theta$. If not, then $G$ contains $p$
satisfying $i)$. So no $p \in G$ satisfies $ii)$. This also holds
in $N$. In $N$, take $p$ to be an initial segment of $G$ of
non-standard length. The failure of $ii)$ for that $p$ produces an
$n$ and $\alpha$ which index a node at which $\phi \rightarrow
\theta$ is false, showing $\bot \not \models \phi \rightarrow
\theta.$

$\exists x \; \phi(x)$: Trivial.

$\forall x \; \phi(x)$: As in the cases for = and $\rightarrow.$

\end {proof}

\begin {Lem}
$\bot \models$ \FANc
\end {Lem}

\begin {proof}
Suppose that at $\bot$ we have a decidable set $C' \subseteq 2^*$
inducing a $c$-bar $C$. We would like to show that at $\bot$ the
$c$-bar $C$ is uniform, which means that, for some $k, \; C$
contains every sequence of length at least $k$; in notation, $C
\supseteq 2^{\geq k}$. This is equivalent with $C'$ containing
$2^{\geq k}$, which is what we will prove.

Say that $\alpha \in 2^*$ is {\em good} if there is a natural
number $k$ such that, whenever $n \geq k$ and $\beta \supseteq
\alpha$ has length at least $k$, $\neg \langle n, \beta \rangle
\Vdash^* C' \supseteq 2^{\geq k}.$ Observe that if $\alpha^\frown
0$ and $\alpha^\frown 1$ are good then so is $\alpha$ (by taking
$k$ sufficiently large). So if the empty sequence $\langle
\rangle$ is bad (i.e. not good) then there is a branch $Br_0$ of
bad nodes. For each $\alpha \in Br_0$, by the definition of
badness, taking $k$ to be the length $\mid\alpha\mid$ of $\alpha$,
we have some $\beta \supseteq \alpha$ and $n \geq \mid \alpha
\mid$ such that $\neg \langle n, \beta \rangle \not \Vdash^* C'
\supseteq 2^{\geq k}.$ Because $\neg \langle n, \beta \rangle
\Vdash^* \phi$ is defined as the truth of $\phi^{\neg \langle n,
\beta \rangle}$ in the classical universe $V$, we can reason
classically and conclude that there is a $\gamma \in 2^{\geq k}$
such that $\neg \langle n, \beta \rangle \Vdash^* \gamma \not \in
C'.$ By choosing $\alpha$'s of increasing length, we can get
infinitely many $\gamma$'s of increasing length, in particular
infinitely many distinct $\gamma$'s. Hence there is a branch
$Br_1$ such that each node in $Br_1$ has infinitely many such
$\gamma$'s as extensions.

That was all in $M$. Now with reference to $N$, if $\alpha \in
Br_0^N$ has standard length, then the corresponding choice of
$\gamma$ is also standard, since it's the same $\gamma$ as in $M$.
So if we choose a non-standard $\gamma$ coming from the procedure
above, that $\gamma$ came from a non-standard $\alpha$. Since $N
\models ``Br_1$ is infinite," there is a non-standard node on
$Br_1^N$, with some such $\gamma$ as an extension; since the node
chosen from $Br_1^N$ was non-standard, so is $\gamma$, and hence
so is the $\alpha$ that $\gamma$ came from. From $\alpha$, we also
have $\beta \supseteq \alpha$ and $n \geq \mid \alpha \mid$ with
$\neg \langle n, \beta \rangle \Vdash^* \gamma \not \in C'$. In
particular, $\langle n, \beta \rangle$ indexes a node in the
model. But at $\bot$, $C'$ induces a $c$-bar, so $\bot \models
``$there is a node $\delta \in Br_1$ such that $\delta \in C$;
that is, every extension of $\delta$ is in $C'$." This contradicts
the choice of $\gamma$.

We conclude from this that $\langle \rangle$ is good. Fix $k$
witnessing this goodness. We will show $\bot \models C' \supseteq
2^{\geq k}$.

First, if $\delta \in 2^{\geq k}$ is standard, then, for any $n$
and $\beta$ non-standard, $\neg \langle n, \beta \rangle \Vdash^*
\delta \in C',$ so, with reference to the Kripke node $\langle n,
\beta \rangle,$ $\langle n, \beta \rangle \models \delta \in C'.$
Since $C'$ is decidable, $\bot \models \delta \in C'.$

To finish the argument, we need only consider nodes $\langle n,
\beta \rangle,$ and show $\langle n, \beta \rangle \models C'
\supseteq 2^{\geq k}.$ If $\beta$ has length at least $k$, this
follows from the goodness of $\langle \rangle.$ The only other
case is $\beta$ of length less than $k$ such that $G(\beta) =
\infty$. It suffices to show that, for any such fixed $\beta$, in
$M$ there is a finite $n$ such that, for all $m \geq n, \; \neg
\langle m, \beta \rangle \Vdash^* C' \supseteq 2^{\geq k}$.

Toward that end, suppose not. Then for infinitely many $m$ there
is a $\gamma$ of length at least $k$ such that $\neg \langle m,
\beta \rangle \Vdash^* \gamma \not \in C'$. If those $\gamma$'s
are of bounded length then one occurs infinitely often. For that
fixed $\gamma$, by overspill there is a non-standard $m$ such that
$\neg \langle m, \beta \rangle \Vdash^* \gamma \not \in C'$. But
$\langle m, \beta \rangle$ is a Kripke node, and we already saw
that, for $\delta \in 2^{\geq k}, \; \bot \models \delta \in C'$,
which is a contradiction. Hence there are infinitely many
different $\gamma$'s. That means there is a branch $Br_2$ such
that every node on $Br_2$ has infinitely many different $\gamma$'s
as extensions. Pick a non-standard $m$ such that the corresponding
$\gamma$ extends a non-standard node of $Br_2$. But again, $\bot
\models ``C$ is a $c$-bar," so $\bot \models ``$there is a node
$\delta \in Br_2$ such that $\delta \in C$; i.e. every extension
of $\delta$ is in $C'$." This contradicts the choice of $\gamma$.
\end {proof}

\begin {Lem}
$\bot \models IZF$
\end {Lem}

\begin {proof}
As before, all of the axioms have soft proofs, save for
Separation. Given $\phi$ and $\sigma$, let Sep$_{\phi, \sigma}$ be
$\{ \langle B_i \cup B_p, \tau \rangle \mid \langle B_i, \tau
\rangle \in \sigma$ and $p \Vdash^* \phi(\tau) \} \cup \{ \langle
B, \tau \rangle \mid$ for some $\neg \langle n, \alpha \rangle \in
B$ and some $B_i, \; \langle B_i, \tau \rangle \in \sigma, \; \neg
\langle n, \alpha \rangle \Vdash^* B_i$, and $\neg \langle n,
\alpha \rangle \Vdash^* \phi(\tau) \}.$ By lemma \ref{truth2},
this works.
\end {proof}

\section {\FANP does not imply \FANF}

As usual, let $G$ be generic as above. In $M[G]$, the Kripke model
will have bottom node $\bot$, and successor nodes labeled by those
$\alpha \in 2^*$ such that $G(\alpha) = \infty$. As is standard,
terms are defined inductively, and always subject to the usual
restrictions so as to have a Kripke model. That is, to define the
full model \cite {L} over any partial order $\langle P, <
\rangle$, at node $p \in P$ a term $\sigma$ is any function with
domain $P^{\geq p}$ such that $\sigma(q)$ is a set of terms at
node $q$; furthermore, with transition function $f_{qr}$ for
$q<r$, if $\tau \in \sigma(q)$ then $f_{qr}(\tau) \in \sigma(r)$;
finally, $f_{pq}$ is extended to $\sigma$ by restriction:
$f_{pq}(\sigma) = \sigma \upharpoonright P^{\geq q}$. That is
called the full model, because everything possible is being thrown
in. For the current construction, we will take a sub-model of the
full model by imposing one additional restriction: a term at any
node $\alpha$ other than $\bot$ must be in the ground model $M$.

Let $C$ be the term such that $\bot \models ``\hat{\beta} \in C"$
($\beta \in 2^*$) iff, for some initial segment $\beta
\upharpoonright n$ of $\beta, \; G(\beta \upharpoonright n)$ = IN,
and, at node $\alpha \not = \bot, \; \alpha \models ``\hat{\beta}
\in C"$ iff $\beta$ is not an initial segment of $\alpha$.

\begin {Lem} $\bot \models \FANP$
\end {Lem}

\begin {proof} If $\bot \models ``B \subseteq 2^*$ is decidable"
then, for any $\beta \in 2^*, \; \bot \models ``\hat{\beta} \in
B"$ iff, for some node $\alpha \not = \bot, \; \alpha \models
``\hat{\beta} \in B"$ iff the same holds for all $\alpha \not =
\bot.$ Hence $\bot \models ``B = \hat{B_M}"$ for some set $B_M \in
M$. So if $\bot \models ``B_n$ is a sequence of decidable trees,"
then that sequence is the image of a sequence of sets from $M$.
Hence their intersection internally is the image of a set from
$M$. So if $\bigcap_n B_n$ is internally a bar, it is the image of
a bar, and by the Fan Theorem in $M$ is uniform.
\end {proof}

\begin {Lem} $\bot \not \models \FANF$
\end {Lem}

\begin {proof} At $\bot, \; C$ is not uniform, so it suffices to
show $\bot \models ``C$ is a bar". If $\bot \models ``P$ is a path
through $2^*$" then $\bot \models ``P$ is decidable", and as above
$P$ is then the image of a ground model path. Generically, for
some $\beta$ along that path, $G(\beta)$ = IN. For that $\beta, \;
\bot \models ``P$ goes through $\hat{\beta}$ and $\hat{\beta} \in
C."$
\end {proof}

\begin {Lem} $\bot \models IZF$
\end {Lem}

\begin {proof} Not only are most of the axioms trivial to verify,
in this case even Separation is too. Given a formula $\phi$, term
$\sigma$, and node $nd$, let Sep$_{\phi, \sigma}(nd)$ be $\{ \tau
\mid \tau \in \sigma$ and $nd \models \phi(\tau) \}.$ The reason
that at node $\alpha$ this is in $M$ is that, at $\alpha$,
$\phi$'s parameters can also be interpreted in $M$, and so truth
at $\alpha$ is definable in $M$.
\end {proof}

\section {Questions}
\begin {itemize}
\item We have seen that \FANF holds in every topological model,
and that \FANP holds in the model over any connected Heyting
algebra. Are there any other sufficient or necessary properties
for any of the various fan theorems we have been considering to
hold or fail in a Heyting-valued model?

\item As a particular instance of the previous question, if a
Heyting algebra satisfies \FAND (resp. \FANc), does it
automatically satisfy \FANc (resp. \FANP)?

\item Although we were not able to make use of any Heyting
algebras other than $\Omega$, some seem worthwhile to investigate,
as possibly separating some of these fan theorems, or perhaps
having some other interesting properties. We would include among
these $K(T)$ for various natural spaces $T$, such as $2^\NN$. We
would also include other ways of killing points, such as over a
measure space $\tau$ with measure $\lambda$ modding out by sets of
measure 0:

\[U \sim V \iff \lambda(U ) = \lambda(V) = \lambda ( V \cap U) \]
(two opens are equivalent if their symmetric difference is of
measure zero). The space $\tau / \sim$ should be a Heyting
algebra, which we will denote by analogy with $K$ as $L(\tau)$. Of
particlar interest seem to be $L(I)$ and $L(I \times I)$.

\item In the models presented here, the principles in question
were not true. That's different from their being false (meaning
their negations being true). We expect this could be done by
iterating the constructions presented here. That is, to each
terminal node of the model append another model of the same kind,
starting with the ambient universe of that terminal node as the
new ground model. By iterating this procedure infinitely often,
one is left with a Kripke model with no terminal nodes. In order
still to have a model of IZF, to get the Power Set Axiom for
instance, terms for all of these bars from the iteration might
have to be present at $\bot$, or perhaps some other fix would
work. So this suggestion would at least take some work to
implement, and might even demand some new ideas.

It would be even better, or at least different, if we had a model
with one fixed counter-example. Maybe the models presented here
could be so tweaked. For instance, for \FAND, could we just throw
away the terminal nodes? For \FANc, it might work not to stop a
node just because $\neg b'$ is true, but rather to continue
extending the node to allow finitely many $\neg b'$s to be true.
Or maybe a more radical idea is needed.

\item One of the referees asked about the role of Choice here. It
is not that hard to see that Dependent Choice fails in most (or
all) of these models. Are there some nice choice principles that
are true here? Are there other models in which DC or other choice
principles of interest hold? Are there significant fragments of
Choice that are incompatible with these separations?

\item Within reverse classical mathematics, many weakenings of
Weak K\"onig's Lemma (classically equivalent to the Fan Theorem)
have been identified. Of interest to us here is Weak Weak
K\"onig's Lemma. Whereas WKL states that any bar (closed under
extension, for simplicity) contains an entire level of $2^*$, WWKL
states that any bar contains half of a level. (WWKL has been shown
to be connected to the development of measure theory.) In our
context, any of the principles we have been considering could be
so weakened, yielding Weak \FAND, Weak \FANc, Weak \FANP, and Weak
\FANF. Clearly any principle implies its weak correlate (e.g.
\FANP implies Weak \FANP), and any weak principle implies the weak
principles lower down (e.g. Weak \FANP implies Weak \FANc),
forming a bit of a square. Are there any implications along the
diagonal (e.g. between \FANc and Weak \FANP)? Are these weak
principles even natural or interesting, by being equivalent with
interesting theorems?

\item Are there any other interesting principles to be found here,
for instance $\Pi^0_n$-FAN for $n > 1$, or adaptations of reverse
math principles beneath WKL other than WWKL?

\end{itemize}

\end{document}